\documentclass[11pt]{amsart}
 
\usepackage{amsmath,amsthm, amsfonts, amsxtra, amssymb, amscd}

\usepackage[english]{babel} 
\usepackage[all]{xy}
\usepackage[backref, colorlinks, linktocpage, citecolor = blue, linkcolor = blue]{hyperref}

\usepackage{graphics, epic}
 \usepackage{graphicx}
\usepackage{multicol}
 \usepackage{tikz}
\usepackage{psfrag}
 \usepackage{young}
 \usepackage{ytableau}
\usepackage[enableskew]{youngtab}
\usepackage{tikz}

\usepackage{mathtools}  
\usepackage{bm}         
\usepackage{dsfont}	
\usepackage{ytableau}   
\usepackage{microtype}  
\usepackage{placeins}

\newtheorem*{lemma*}{Lemma}
\newtheorem{lemma}[subsection]{Lemma}
\newtheorem*{theorem*}{Theorem}
\newtheorem{theorem}[subsection]{Theorem}
\newtheorem*{proposition*}{Proposition}
\newtheorem{proposition}[subsection]{Proposition}

\newtheorem*{corollary*}{Corollary}

\newtheorem{corollary}[subsection]{Corollary}

\theoremstyle{definition}
\newtheorem*{definition*}{Definition}

\newtheorem*{example*}{Example}

\theoremstyle{remark}
\newtheorem*{remark*}{Remark}
\newtheorem{remark}[subsection]{Remark}
 \newtheorem{definition}[subsection]{Definition}

\renewcommand{\phi}{\varphi}

\newcommand{\be}{\begin{enumerate}}
\newcommand{\ee}{\end{enumerate}}
\DeclareMathOperator{\tr}{tr}

\newcommand{\one}[0]{\mathds{1}}

\title{A construction of swap or switch polynomials}
\author{Claudio Procesi} 

\begin{document}\date{Version from 28.02.2021}
 
\address{Dipartimento di Matematica, G. Castelnuovo,
Universit\`a di Roma La Sapienza, piazzale A. Moro,  00185,
Roma, Italia}
\email{procesi@mat.uniroma1.it}
 \maketitle

\begin{abstract}
We discuss several constructions of {\em swap polynomials}  that is  non commutative  polynomials in matrix variables $x_i\in M_d(\mathbb Q)$  with values  in $M_d(\mathbb Q)^{\otimes 2}$ which are multiples of the transposition operator $(1,2)$.
\end{abstract}\tableofcontents
\section{Introduction}
Given a positive integer $d$  and a field $F$ denote by $M_d(F)$ the algebra of $d\times d$ matrices with entries in $F$.  We will assume that $F$ has characteristic 0  and in fact work with $F=\mathbb Q$ and then denote  $M_d:=M_d(\mathbb Q).$

Denote by $F\langle X\rangle=F\langle x_1,\ldots,x_i,\ldots\rangle$ the free algebra in the variables $x_i$.  The elements of $F\langle X\rangle$ are usually called {\em non commutative polynomials}.
  Given an associative algebra $A$ over $F$   by definition the homomorphisms of $F\langle X\rangle$ to $A$  correspond to maps $X\to A$ and can be thought of as {\em evaluations} of the variables $X$ in $A$.   
%
Given any positive integer $k$ we may consider the algebra $A^{\otimes k}$ and then for every evaluation  $\pi:F\langle X\rangle\to A$ of $X$ in $A$ we have a corresponding evaluation  $\pi^{\otimes k}:F\langle X\rangle^{\otimes k}\to A^{\otimes k}$.
\begin{definition}\label{tPI}
The elements of $F\langle X\rangle^{\otimes k}$ will be called {\em tensor polynomials}.  
They can be thought of as (non commutative) polynomials in the {\em tensor  variables}  $x_j^{(i)}:=1^{\otimes i-1}\otimes x_j  \otimes  1^{\otimes  k-i }$.
An element $f\in F\langle X\rangle^{\otimes k}$ is called a {\em tensor  polynomial identity} for $A$  (short a TPI)  if it vanishes under all evaluations of $X$ in $A$.
\end{definition} 

%
   We then devote our study to {\em permutation valued polynomials} and in particular to {\em swap polynomials}. \smallskip

\begin{definition}\label{swa}
Consider a   non commutative $2$--tensor polynomial $f\in F\langle X\rangle^{\otimes 2}$ in   variables    $x_i$. Then $f$   is called a {\em swap polynomial} for $d\times d$ matrices  if, when evaluated  in the algebra of $d\times d$ matrices     it is not a TPI and  takes values only multiples of the exchange (or swap)  operator $(1,2):a\otimes b\mapsto b\otimes a$, $(1,2)\in M_d^{\otimes 2}$.

A  $2$--tensor polynomial  $f=\sum_i A_i\otimes B_i$ is {\em balanced}  if all $A_i, B_i$ are homogeneous of the same degrees. \end{definition}

This notion appears in a recent preprint  {\em Translating Uncontrolled Systems in Time,} by  David Trillo, Benjamin Dive, and Miguel Navascu\'es,\ arXiv:1903.10568v2 [quant--ph]  28 May 2020 \cite{TDN}. The authors introduce  some special tensor valued polynomials.  In particular they prove the existence of such operators using the classical theory of central polynomials  developed independently by Formanek and Razmyslov.  The resulting swap polynomials have usually very large degrees, although they exhibit such a polynomial for 2  tensor $2\times 2$ matrix variables   of degree 10. In their notations the degree is 5, or $5+5$, which according to their computations is the least degree for a balanced swap polynomial. 
They consider a tensor variable $x_i\otimes x_j$  of degree 1  while in this paper we consider it of degree 2 in the matrix variables, so their polynomial is of degree 5 in the tensor variables.
\medskip

Swap polynomials       belong to the interesting class of {\em permutation valued polynomials} that is polynomials whose values are a scalar times a constant linear combination of permutation operators. 

 The existence of such polynomials is a consequence of the existence of central polynomials 
and  follows easily from the so called {\em Goldman elements} as in Saltman, \cite{salt}, see Theorem \ref{Gol}. 
\smallskip

We discuss first  the case of $2\times 2$ matrices and two variables, which can be described almost completely. Then, for general $d\times d$ matrices
we   propose a   canonical construction of a  swap polynomial of degree $2d^2$   in $2d^2$  variables.

In particular we  exhibit a  swap polynomial of degree 8 (or 4 in their language) but in 8 variables for $2\times 2$ matrices.

The construction is based on a   canonical central polynomial proposed by Regev and proved to be non zero by Formanek. The advantage of these polynomials is that they are alternating in two sets of $d^2$  variables which implies that we know an explicit formula for their values.

\smallskip

The paper  \cite{TDN} was pointed out to me by Felix  Huber while discussing his  recent work 
{\em Positive maps and trace polynomials from the Symmetric Group}  arXiv:2002.12887 28--2--2020.\smallskip

\section{Tensor Polynomials}

\subsection{Azumaya algebras,  the     {\em Goldman element}\label{bira}} The following Definition and Theorem is attributed, with no reference,  to Oscar Goldman in  the book of M. A. Knus, M. Ojanguren,  \textit{Th\'eorie de la descente et alg\`ebres d'Azumaya} page 112 
\cite{KnusO}.\smallskip

Let $R$ be a rank $n^2$ Azumaya algebra over its center $A$. By definition of Azumaya algebra  the map $$\pi:R\otimes _AR^{op}\to End_A(R),\ \pi(
\sum_ia_i\otimes b_i)(x)=\sum_ia_ixb_i$$ is an isomorphism, moreover there is a faithfully flat extension $A\to B$ so that $B\otimes_AR\simeq M_n(B)$  and the trace of $ M_n(B)$ restricted to $R$ takes values in $A$, so we have an $A$--linear map $tr:R\to A\subset R$.
\begin{definition}\label{goele}
We define  the     {\em Goldman element}     $\mathtt t\in R\otimes _AR$ by $\boxed{\pi(\mathtt t)(x):=tr(x)}.$ 
\end{definition}
\begin{theorem}\label{Gol} The  Goldman element satisfies 
\begin{equation}\label{Gol1}
\mathtt t^2=1,\quad \mathtt t(a\otimes b)\mathtt t^{-1}=b\otimes a.
\end{equation} 
\end{theorem}   
 \begin{proof} 

Under the faithfully flat extension $A\to B$ the element $\mathtt t$ must map to $  \sum_{i,j=1}^ne_{i,j}\otimes e_{j,i}$, by uniqueness since this element satisfies the same property:
$$\sum_{i,j=1}^ne_{i,j}e_{h,k}e_{j,i} =\begin{cases}
0\ \text{if}\ h\neq k\\
1 \ \text{if}\ h= k
\end{cases}.$$ Then the properties of Formula \eqref{Gol1} can be verified directly:
$$\mathtt t^2=\sum_{i,j=1}^ne_{i,j}\otimes e_{j,i}\sum_{h,k=1}^ne_{h,k}\otimes e_{k,h} =\sum_{i,j=1}^ne_{i,i}\otimes e_{j,j}=1.$$
The second  property can again be verified in the split algebra and it is
$$(\sum_{i,j=1}^ne_{i,j}\otimes e_{j,i})e_{a,b}\otimes e_{c,d}(\sum_{h,k=1}^ne_{h,k}\otimes e_{k,h})=e_{c,d}\otimes  e_{a,b}. $$

\end{proof}By Corollary 10.4.3  of \cite{agpr} we can take as $B=A_n(R)$ the  commutative ring giving the universal  map into matrices that is for any  map $j:R\to M_n(C)$ one has a map $\bar j:A_n(R)\to C$ making commutative the diagram.
\begin{equation}\label{codd}
\xymatrix{
R\ar[r]^{
\!i\ \quad  }\ar[rd]_j
& M_n(A_n(R))  \ar[d]^{M_n(\bar j)}\\
&M_n(C)}\end{equation}It follows that if $\rho:R\to S$ is a ring homomorphism of rank $n^2$ Azumaya algebras  we have
\begin{equation}\label{codd1}
\xymatrix{
R\ar[r]^{
\!i_R\ \quad  }\ar[d]_\rho
& M_n(A_n(R))  \ar[d]^{M_n(\bar i_S\circ\rho)}\\
S\ar[r]^{
\!i_S\ \quad  }& M_n(A_n(S))}\end{equation}and $\mathtt t_R,\ \mathtt t_S$ the respective Goldman elements we have $\rho(\mathtt t_R)=\mathtt t_S$, hence:
\begin{corollary}\label{ges}
Under any splitting $C\otimes_AR\simeq M_n(C)$,   the element $\mathtt t$ maps to the switch operator on $C^n\otimes C^n$. 
\end{corollary}
\begin{proof}
The element  $\mathtt t$ maps to $\sum_{i,j=1}^ne_{i,j}\otimes e_{j,i}$ which is the switch operator on $C^n\otimes C^n$ since
$$\sum_{i,j=1}^ne_{i,j}\otimes e_{j,i}(e_{a }\otimes e_{b}) =e_{b }\otimes e_{a} .$$
\end{proof}
\begin{remark}\label{nou}
The properties of Formula \eqref{Gol1}  determine $\mathtt t$ only up to a multiplicative scalar  $\alpha$ with $\alpha^2=1$.

In particular if $A$ is a domain $\alpha=\pm 1$.\end{remark}
If  $R$  is a free rank $n^2$  module over $A$ with basis $a_1,a_2,\ldots,a_{n^2}$  then there is a unique dual basis for the trace form $tr(xy)$. That is, there are unique elements  $a_1^*,a_2^*,\ldots,a_{n^2}^*$ with $tr(a_ia_j^*)=\delta_i^j$.  Then we have   \begin{equation}\label{ledua}
\mathtt t=\sum_{i=1}^{n^2} a_i\otimes a_i^*.
\end{equation}
This depends upon the fact that the element $\sum_{i=1}^{n^2} a_i\otimes a_i^* $ is independent of the  basis chosen.  
For $R=M_n(A)$  the dual basis  of   elementary matrices is $e_{i,j}^*=e_{j,i}$. So under the faithfully flat splitting the element $\sum_{i=1}^{n^2} a_i\otimes a_i^*$  coincides with $\sum_{i,j=1}^ne_{i,j}\otimes e_{j,i}$.

\begin{remark}
In the Physics literature for $n=2$ and $A=\mathbb C$ one has the {\em Pauli matrices} which form, up to the normalizing factor $\frac 1{\sqrt{2}}$,  an orthonormal basis for the trace form  restricted to Hermitian matrices where it is positive.
\end{remark}\smallskip

\subsection{The case of generic matrices}
In the Theory of algebras with polynomial identities, see  \cite{agpr}, one has the basic    algebra $R_{k,d}:=F[\xi_1,\ldots,\xi_k]$  of polynomials in $k$ generic $d\times d$ matrices $\xi_i=(\xi_{h,k}^{(i)})$ which one identifies with the {\em non commutative polynomial functions in $k$ matrix variables}. Here the variables $\xi_{h,k}^{(i)}$ are the coordinates of  the $kd^2$ vector space $M_d(F)^k$. See \cite{agpr} for details. This algebra is a domain with a center $Z$. If $G$ is the field of fractions of $Z$  we have, as soon as $k>1$, that $R_{k,d}\otimes _ZG:=D_{k,d}$ is a division algebra of dimension $d^2$ over its center $G$.  Moreover $G$ is the field of rational functions on      $M_d^k$  invariant   under conjugation action by $GL(d,F)$ the group of invertible matrices. Finally the polynomial  functions on      $M_d^k$  invariant   under conjugation   are generated by the traces of the monomials in the $\xi_i$. Thus we have the Goldman element  $\mathtt t\in D_{k,d}\otimes_GD_{k,d}$.

Notice that this element is {\em independent} of the variables $\xi_i$ in particular we may  find it in the algebra of just two variables.\smallskip  

In general a rational function $f\in D_{k,d}$ can be evaluated on an open set of matrices and on this set the evaluation of  $\mathtt t$ is the switch operator  $(1,2)$.  To find a swap polynomial is equivalent to find a common denominator   $c$ for $\mathtt t$   which is a scalar polynomial, that is a {\em central polynomial}. In fact this can be done in several ways.\smallskip

The algebra $R_{k,d}:=F[\xi_1,\ldots,\xi_k],\ k\geq 2$ has a non trivial center due to Formanek \cite{formanek3} and Razmyslov  \cite{razmyslov2}. If $c\in F[\xi_1,\ldots,\xi_k]$ is an element of the center with no constant term then inverting $c$ one has that the algebra 
$ F[\xi_1,\ldots,\xi_k][c^{-1}]$ is Azumaya of rank $n^2$,  
Corollary 10.3.5 of \cite{agpr}. Thus it has a Goldman element which in fact coincides with that  $\mathtt t$   of  $D_{k,d}$. 

Thus  for each such $c$  there exist $a_i,b_i\in  F[\xi_1,\ldots,\xi_k]$ and $h>0$ so that  we have   $\mathtt t=c^{-h} \sum_ia_i\otimes b_i,\ a_i, \ b_j\in   F[\xi_1,\ldots,\xi_k]$.  In other words,    by adding an extra variable $\zeta$  we have the identity 
\begin{equation}\label{trac}
c^{ h} \mathtt t=\sum_ia_i\otimes b_i,\quad \sum_ia_i\zeta  b_i=c^h tr(\zeta).
\end{equation} The element  $\sum_ia_i\otimes b_i$ is thus a swap polynomial.
\begin{theorem}\label{cca}
A tensor polynomial $\sum_ia_i\otimes b_i$ is  a swap polynomial if and only if adding a variable $\zeta$ we have that 
$\sum_ia_i\zeta b_i$ is a central polynomial  which vanishes for $\zeta$ with $tr(\zeta)=0$.\end{theorem}
\begin{proof}
If $\sum_ia_i\otimes b_i$ is  a swap polynomial then it is   equal, as function on matrices, to $\alpha (1,2)=\alpha \sum_{i,j=1}^ne_{i,j}\otimes e_{j,i} $ with $\alpha$ an invariant  scalar function. Then
$$\sum_ia_i\zeta b_i= \alpha \sum_{i,j=1}^ne_{i,j}\zeta e_{j,i}=\alpha\cdot  tr(\zeta)1. $$
Conversely if $\sum_ia_i\zeta b_i=\beta 1$ is a central polynomial, $\beta$ some polynomial invariant,  which vanishes for $\zeta$ with $tr(\zeta)=0$ then  $\beta=tr(\zeta)\alpha $ is divisible by $tr(\zeta)$. Since $\beta$ is linear in $\zeta$  we have $\alpha $ is independent of $\zeta$ and $\sum_ia_i\zeta b_i=tr(\zeta)\alpha 1$ implies 
$$\sum_ia_i\otimes b_i=\alpha\mathtt t. $$
\end{proof}
\begin{remark}\label{cenn}
For a swap  polynomial $\sum_ia_i\otimes b_i=\alpha \mathtt t$    we have  $n\alpha=\sum_ia_ib_i$ so as soon as the characteristic does not divide  $n$ we also have that $n \alpha=\sum_ia_ib_i$ is a central polynomial.
\end{remark}
\section{$2\times 2$  matrices}  \subsection{Swap polynomials in two variables}
In \cite{TDN}  the authors explain a method to construct balanced swap polynomials, Definition \ref{swa}. The condition of being balanced is necessary for their applications.   They exhibit one, let us denote it by $Q(x,y)$, in two variables  $x,y$ for $2\times 2 $ matrices, involving 40 terms and of degree 10 (or 5 +5 in their terminology). 
\begin{equation}\label{qab}
Q(x,y):=
\end{equation}
$$x y^2 x  y   \otimes  x y^2 x  y - 
 x y^2 x  y   \otimes y^2 x^2  y - 
 x y^3 x \otimes  x y^2 x  y + 
 x y^3 x \otimes  x y^3 x + 
 x y^3 x \otimes  y  x  y  x  y$$$$ - 
 x y^3 x \otimes y^2 x  y  x - 
 x y^4  \otimes  x  y  x  y  x + 
 x y^4  \otimes  y  x^2  y  x - 
 y  x y^2 x \otimes  x y^2 x  y - 
 y  x y^2 x \otimes  x y^3 x $$$$+ 
 y  x y^2 x \otimes  y  x  y  x  y + 
 y  x y^2 x \otimes y^3 x^2 - 
 y  x y^3  \otimes  x  y  x  y  x + 
 y  x y^3  \otimes  x y^2 x^2 + 
 y  x y^3  \otimes  y  x  y  x^2 $$$$- 
 y  x y^3  \otimes y^2 x^3+ 
y^2 x  y  x \otimes  x y^2 x  y - 
y^2 x  y  x \otimes  x y^3 x - 
y^2 x  y  x \otimes  y  x  y  x  y + 
y^2 x  y  x \otimes  y  x y^2 x $$$$- 
y^2 x y^2  \otimes  x  y  x^2  y + 
y^2 x y^2  \otimes  x  y  x  y  x + 
y^2 x y^2  \otimes  y  x^3 y - 
y^2 x y^2  \otimes  y  x^2  y  x + 
y^3 x^2   \otimes  x y^3 x $$$$- 
y^3 x^2   \otimes  y  x  y  x  y + 
y^3 x^2   \otimes y^2 x^2  y - 
y^3 x^2   \otimes y^3 x^2 + 
y^3 x  y   \otimes  x  y  x  y  x - 
y^3 x  y   \otimes  x y^2 x^2$$$$ - 
y^3 x  y   \otimes  y  x  y  x^2 + 
y^3 x  y   \otimes y^2 x^3- 
y^4 x \otimes  x^2 y^2 x + 
y^4 x \otimes  x  y  x^2  y + 
y^4 x \otimes  x  y  x  y  x $$$$- 
y^4 x \otimes  y  x^3 y - 
y^4 x \otimes  y  x  y  x^2 + 
y^4 x \otimes y^2 x^3+ 
y^5   \otimes  x^2  y  x^2 - 
y^5   \otimes  x  y  x^3$$ This   has been built by a computer program, and I have verified it,    and now I want to give a  theoretical explanation for its existence and that of many other swap polynomials, Theorem \ref{esss}. Its value, checked by Computer,  is 
 \begin{equation}\label{val}
tr(y)^2\det([x,y]) ^2\mathtt t= tr(y)^2 [x,y]  ^4\mathtt t.
\end{equation}   
\paragraph{Formulas}  This topic is treated in detail in Chapter 9   of \cite{agpr}. We want to recall some formulas which will be useful to understand swap polynomials.

The case of two $2\times 2$ matrices $x,y$ is fully treated in the 1981 paper of Formanek, Halpin \cite{Foli} and will be quickly reviewed here. Start with  Exercise 9.1.1    of \cite{agpr}. \begin{proposition}\label{FolI}
The ring $T$ of invariants of 2, $2\times 2$ matrices $x,y$ is the polynomial ring in 5 generators
 \begin{equation}\label{invdgm}
 tr(x),\ \det(x),\  tr(y),\ \det(y),\  tr(xy).
\end{equation}   The ring  $S$ of  equivariant maps of 2, $2\times 2$ matrices $x,y$ to $2\times 2$  matrices (also called the {\em trace algebra}) is a free module $S=T+Tx+Ty+Txy$ over the ring of invariants  with basis
 \begin{equation}\label{basdgm}
1,\ x,\ y,\ xy.
\end{equation}  The multiplication table is given by Cayley--Hamilton:
$$ x^2=tr(x)x-\det(x),\ y^2=tr(y)y-\det(y),$$\begin{equation}\label{polCH}
 yx=-xy+tr(x)y+tr(y)x+tr(xy)-tr(x)tr(y).
\end{equation}
\end{proposition}
Of course in characteristic 0 we may replace the generators  of Formula  \eqref{invdgm} with
 \begin{equation}\label{invdgm1}
 tr(x),\ tr(x^2),\  tr(y),\ tr(y^2),\  tr(xy).
\end{equation} 
Let $R:=F\langle x,y\rangle\subset S$ be the subalgebra generated by the two {\em generic matrices} $x,y$. This is also the free algebra in two variables modulo the ideal of polynomial identities of $2\times 2$ matrices.

The structure of $R$ is  deduced in   \cite{Foli} from the following identities:
\begin{equation}\label{ii}
\begin{matrix}
[x,y]x=tr(x)[x,y]-x[x,y],&[x,y]y=tr(y)[x,y]-y[x,y]\\\\\det(x)[x,y] =x[x,y]x ,&\det(y)[x,y] =y[x,y]y\\ \\tr(xy)[x,y]= xyxy-yxyx.
\end{matrix}
\end{equation} From which it follows (Lemma (2) of \cite{Foli}):
\begin{proposition}\label{comm}
The commutator ideal  $R[x,y]R$   equals  $S[x,y]=S[x,y]S$.
\end{proposition} From this Theorem (3) of \cite{Foli} follows.
\begin{theorem}\label{dec} We have the decomposition of $R$ as vector space:
\begin{equation}\label{dee}
R=\oplus_{i,j} Fx^iy^j\oplus S[x,y].
\end{equation}

\end{theorem}

Recall that, given an inclusion $A\subset B$ of rings, the {\em conductor} (of $A$ in $B$) is the maximal ideal $I$ of $A$ which is also an ideal of $B$. In other words it is the set of $a\in A$ with $Ba+aB\subset A$.
\begin{proposition}\label{con}
The commutator ideal  $R[x,y]R$   is the conductor of $R$ in $S$.
\end{proposition}
\begin{proof}
By the previous proposition $R[x,y]R$ is an ideal in $S$ so it is contained in the conductor. Now the conductor is a bigraded ideal  and $R/R[x,y]R$ has a basis of the ordered monomials $x^iy^j$. So if we had an element in the conductor not in $R[x,y]R$ we would have that one of those monomials is in the conductor. If $x^iy^j$  is in the conductor we have  $x^iy^jtr(x)=f(x,y)$  for some non commutative polynomial. Setting $y=1$  we obtain some identity $ x^i tr(x)=\alpha x^{i+1},\ \alpha\in F$. This implies $tr(x)=\alpha x$  a contradiction.
\end{proof}
The aim of the paper \cite{Foli} was in particular to compute the Poincar\'e series  $\mathcal  P(R)=\sum_{i,j=1}^\infty \dim(R_{i,j})t^is^j$  where $\dim(R_{i,j})$ is the dimension   of $R$ in bidegree $i,j$ with respect to $x,y$. Now we clearly have
$$
\mathcal  P(T)=\frac 1{(1-t)(1-s)(1-t^2) (1-s^2)(1-ts)  } , $$$$ \mathcal  P(S)=( 1+ t+s+ts)\mathcal  P(T).$$\begin{equation}\label{poin}\mathcal  P(R)= \frac 1{(1-t)(1-s)}+ts  \mathcal  P(S) , \ \mathcal  P(Z)=1+(s^2t^2)\mathcal  P(T) . 
\end{equation} Here $Z$ is the center of $R$  and the last Formula follows from Theorem \ref{cd}.

Finally the Poincar\'e series of the free algebra is $(1-s-t)^{-1}$  so the Poincar\'e series of the polynomial identities in 2 variables for $2\times 2$ matrices is, Theorem (10) of \cite{Foli}:
$$ s^2t^2(s+t-st)(1-s)^{-2}(1-t)^{-2} (1-st)^{-1}(1-s-t)^{-1}.  $$ \smallskip

In fact it is important to describe some basic elements in the conductor of the ring $R_n$ of generic $2\times 2$ matrices in $n\geq 3$  variables  inside the corresponding trace ring $S_n$.

\begin{lemma}\label{com}
An element $f(x_1,\ldots,x_n)\in R_n$ is in the conductor of $S_m$ for all $m\geq n$  if and only if, adding  an  extra variable $x_{n+1}$, we have  that $tr(x_{n+1})f(x_1,\ldots,x_n)\in R_{n+1}.$
\end{lemma}
\begin{proof}  
Write the polarized form of the  Cayley--Hamilton identity in the form
$$  tr(z) tr(w)= - zw-wz+tr(z)w+tr(w)z+tr(zw).$$
One has recursively that  $\prod_{i=1}^mtr(z_i)=\sum_j A_jtr(B_j)$ with $A_j,B_j$ explicit non commutative polynomials. 

The algebra $S_m$ is generated over $R_m$ by the elements  $tr(M)$  with $M$ a monomial in the generic matrices (of degree $\leq 3$). Therefore  $f(x_1,\ldots,x_n)\in R_n$ is in the conductor of $S_m$ if it {\em absorbs} the elements $\prod_{i=1}^mtr(z_i)$. But by the previous identity one is reduced to a single trace.
\end{proof}
\begin{remark}\label{gecon0} In Theorem 10.4.8 of \cite{agpr} we generalize to generic matrices of any size proving that 
$f(x_1,\ldots,x_n)\in R_n$ is in the conductor of $S_m$ if and only if  for an  extra variable $x_{n+1}$ we have  $\det(x_{n+1})f(x_1,\ldots,x_n)\in R_{n+1}.$                   
\end{remark}
\begin{corollary}\label{gecon}
1)\quad The polynomial $[[ x_1, x_2] , x_3]$  is in the {\em conductor}  of   the generic matrices inside the trace algebra. 

2)\quad  The central polynomial   $[x,y]^2$ is in the {\em conductor}  of the center of the generic matrices inside the invariant ring, that is for all $m$ we have $\prod_{i=1}^mtr(z_i)[x,y]^2$ is also a non commutative central polynomial.

\end{corollary}
\begin{proof}
Recall the basic Formula 9.50 of \cite{agpr}:
\begin{equation}\label{assotr0}
 [z[x_1,x_2]     +   [x_1,x_2] z,x_3 ]=tr(z)  [[ x_1, x_2] , x_3].\end{equation}

With some elementary manipulations one obtains from this 
 
 \begin{equation}\label{assotr}
 tr(z)[x,y]^2=[zx[x,y]     +   x[x,y] z,y ]-x[z[x, y]     +   [x, y] z,y ]=\end{equation}
$$ zx[x,y] y      -yzx[x,y]   -xz[x, y] y  + xyz[x, y] +    [x, y]^2 z .$$

\end{proof}
Recall  also that $[x,y]^2=-\det([x,y])$  is an irreducible polynomial vanishing exactly on the subvariety $V$ of pairs of matrices $x,y$ which are NOT irreducible.  If $c(x,y)$  is any central polynomial in these two variables (with no constant term) then it vanishes on $V$ (Proposition  10.2.2 of \cite{agpr}). Therefore we have 
$c(x,y)=[x,y]^2\alpha,\ \alpha\in T$.  Conversely we have Theorem (5) of \cite{Foli}
\begin{theorem}\label{cd} The center $Z$ of $R$ equals $F+[x,y]^2T$.
\end{theorem}
\begin{proof}
Every  element of the form $c(x,y)=[x,y]^2\alpha,\ \alpha\in T$ can be expressed as a central polynomial by Corollary \ref{gecon} 2).

\end{proof}

Formula \eqref{assotr}  gives, by  Theorem \ref{cca}  \S \ref{bira},  the swap polynomial of  degree 4 in  two  variables:

 \begin{equation}\label{unsw}
P(x,y):= 1\otimes ([x, y]^2 +x[x,y] y )     -y\otimes x[x,y]   
-x\otimes [x, y] y   + xy     \otimes [x, y]  
\end{equation} with value $[x,y]^2\mathtt t$ but quite unbalanced.

Notice that also $P(y,x)=[x,y]^2\mathtt t$ and
 \begin{equation}\label{unsw1}
P(y,x):= 1\otimes ([x, y]^2 -y[x,y] x )     +x\otimes y[x,y]   
+y\otimes [x, y] x   -xy     \otimes [x, y]  
\end{equation} But now from this we can build a balanced swap  polynomial taking the same value as $Q(x,y)$  of Formula  \eqref{qab}.

In fact $Q(x,y)= tr(y)^2 [x,y]  ^2 P(x,y)= tr(y)^2 [x,y]  ^2 P(y,x)$.  The first equals
$$
tr(y)^2[x,y]^4\mathtt t=\begin{matrix}
 tr(y)   [x,y]  ^2\otimes ([x,y]  ^2+x[x,y] y) tr(y)       -[x,y]  ^2y\otimes tr(y)^2x[x,y]\\    - [x,y]  ^2x\otimes tr(y)^2[x, y] y    + tr(y)   [x,y]^2 xy     \otimes tr(y)   [x,y]  \end{matrix} $$ 
similar for the second.
  All terms of these trace tensor polynomials are balanced except the last ones  
  $$tr(y)   [x,y]^2 xy     \otimes tr(y)   [x,y] ,\quad -tr(y)   [x,y]^2yx     \otimes tr(y)   [x,y] ^2$$
  Each of the  balanced terms  containing traces,   both on the left and on the right of $\otimes$ are in the ideal generated by $[x,y]$  so can be written as non commutative polynomials in $x,y$.  As for the last term remark that, developing $[x,y]^2$  as polynomial in the 5 basic generators of the invariants, one has a polynomial of degree $4$ in $x,y$  sum of monomials, each of which can be split as the product of two scalar terms of degree 2.  
   
   Then  we replace $P(x,y)$  with the sum $\frac 12(P(x,y)+P(y,x))$  so that the two unbalanced terms give
   $$ \frac 12tr(y)   [x,y]^3      \otimes tr(y)   [x,y] .$$  In the left term we first replace $ [x,y]^2$ by the polynomial in the basic generators and then move to the right for each monomial one of the two scalar  factors of degree 2. We obtain a balanced trace tensor polynomial of the same type as before  which can be written as a balanced tensor polynomial.  

 \begin{remark}
I have not verified if,  by applying these formulas, one may obtain exactly formula  \eqref{qab}  or another formula giving the same   swap polynomial up to a tensor polynomial identity  (this will depend on which of the monomials we move to the right and  in which order to apply Formulas \eqref{ii}).   Notice that
$$tr(x)tr(y)[x,y]\otimes [x,y]=$$$$(x[x,y]+[x,y]x)\otimes (y[x,y]+[x,y]y)=(y[x,y]+[x,y]y)\otimes (x[x,y]+[x,y]x ).$$ Therefore the way to express a balanced polynomial is not unique.\end{remark}
 
Of course one may also exchange $x$ and $y$. A   more symmetric swap polynomial of degree $5$ both in $x$ and $y$  is, by the same argument,  $
 tr(x) tr(y) [x,y]^4\mathtt t$.   We will refer to these 3 polynomials as $Q_i(x,y),\ i=1,2,3$.

The authors of \cite{TDN}  have in fact verified that there are no balanced swap polynomials in two variables of degree $<10$  and only these 3 in degree 10.  Later we will see that there is a balanced swap polynomial of degree 8 but in 8 variables, Theorem \ref{swaa}.

\begin{theorem}\label{esss} For every invariant $A=A_1A_2$  product of $2h>0 $  factors of degree 1 giving $A_1$  and $k$ factors  of degree 2 giving $A_2$. If either $k=2\ell$ is even, or $k=2\ell+1$ and $h\geq 2$  we have that  $A[x,y]^4\mathtt t$ is the value of a balanced swap polynomial.

\end{theorem}
\begin{proof}
First if $k=2\ell$ we split $A_2=B_1B_2$  each with $\ell$ elements   and  $A_1=C_1C_2C_3$  with $C_1=tr(a)tr(b),\ a,b\in\{x,y\}$ a product  of 2 traces and $C_2,C_3$ of the same degree.

Then $C_1[x,y]^2 P(x,y)=Q_i(x,y)$, depending on the variables appearing in the traces.  Then we multiply by  $B_1B_2C_2C_3$  and we distribute $B_1C_2$ on the first factor and $B_2C_3$ on the second. We obtain a balanced polynomial involving traces but again by the same argument all   terms can be expressed as non commutative polynomials. The second case is similar.
\end{proof}

It remains open the question if one can have a balanced swap polynomial which does not satisfy the previous conditions or even just whose value is not divisible by $[x,y]^4$. My guess is NO. This may be related to the fact that    $-[x,y]^4$  is the {\em discriminant}  of the basis  \eqref{basdgm}. 

The algebra $S$ becomes an Azumaya algebra after inverting the element  $[x,y]^2$ and in fact  
\begin{proposition}\label{Pdb}
The polynomial $P(x,y)$  is also an expression of $[x,y]^2\mathtt t$  by using the dual basis to $1,x,y,xy$ as in Formula \eqref{ledua}.
\end{proposition}
\begin{proof}

In fact, recall from page 374 of \cite{agpr} the   matrix of the trace form, of the basis  \eqref{basdgm},  is: \begin{equation}
D=\begin{vmatrix}
tr(1)&tr(x)&tr(y)&tr(xy)\\
tr(x)&tr(x^2)&tr(xy)&tr(x^2y)\\
tr(y)&tr(yx)&tr( y^2)&tr( xy^2)\\
tr(xy)&tr(x^2y )&tr(xy^2)&tr((xy)^2)
\end{vmatrix} ,\quad \det(D)=-[x,y]^4\end{equation}

  One can compute that the cofactor matrix $
      \bar \Lambda$ of $D$ is divisible by $[x,y]^2$ so is
$ \bar \Lambda=[x,y]^2\Lambda$. 
Setting $$= 
    \det(x)  tr(y) ^2 - 
    tr(x   y) tr(x) tr(y) + tr(x)^2\det[y)-2\det(x)  \det[y) +
      tr(x   y) ^2 := A$$
$$     \Lambda=\begin{vmatrix}
 2A
   & -tr (x) \det (y) & -\det (x) tr (y) & 
  tr (x) tr (y) - tr (xy)\\-tr (x) \det (y) & 2  \det (y) & 
  tr (xy) & -tr (y)\\-\det (x) tr (y) & tr (xy) & 
  2 \det (x) & -tr (x)\\tr (x) tr (y) - tr (xy) & -tr (y) & -tr (x) & 
  2 
\end{vmatrix}$$ 
From this one has that the dual basis, for the trace form of   the basis  \eqref{basdgm}, up to the scalar $[x,y]^2$ is
$$ ( [x,y]^2+ x[ x,y]y),\ -   [x,y] y,\  -  x[x,y],\ [x,y]$$
Then $[x,y]^2\mathtt t$ is given by the dual bases expression \eqref{ledua}
  $$= 1\otimes ( [x,y]^2+ x[ x,y]y)-x\otimes   [x,y] y -y\otimes x[x,y]+xy\otimes [x,y]$$
 is again the polynomial   $P(x,y)$ of Formula \eqref{unsw}.
\end{proof}
   

\medskip 
  
  It remains to exhibit with an explicit formula a balanced swap polynomial  $g(\xi)=f(\xi)(1,2)$ for all $d$.     An explicit construction is performed in \S \ref{Cap} and \S \ref{swap}.

 \section{Swap polynomials}

A general approach  to  balanced swap polynomials is the following.
Start from any  swap trace  polynomial  $$G:=\sum_iA_i\otimes B_i=f(x_1,\ldots,x_n)\mathtt t$$  with $A_i, B_i $ trace polynomials on $d\times d$ matrices,  $f(x_1,\ldots,x_n)$ a scalar invariant function. This can be for instance constructed, for any $n\geq 2$,  by taking $n^2$  monomials $A_i$  in the generic matrices with $\Delta:=\det(tr(A_iA_j))\neq 0$.  Solving for the dual basis (by Cramer's  rule) $B_i=\sum_{i=1}^{d^2} x_{i,j} A_i$ the equations
$$\Delta \delta^i_j=tr(A_jB_i)=  \sum_{i=1}^{d^2} x_{i,j}tr(A_j A_i)\stackrel{\eqref{ledua}}\implies \sum_{i=1}^{n^2}A_i\otimes B_i=\Delta \mathtt t. $$  

One sees first that the {\em homogeneous components} of  $G$ of    degree  $h $   (i.e. $\deg A_i+\deg B_i=h$) are  still  swap trace  polynomials. It is easy to construct from this a  balanced swap polynomial. First by multiplying by a suitable product $\prod_Itr(z_i)$. This can be distributed in each of the two factors to make them balanced.

Next one has  special central polynomials $u(x)$ which are in the conductor  of the inclusion of the ring of central polynomial  inside the ring of invariants. That is if $f(x)$ is any invariant, i.e. a polynomial in the traces we have $u(x)f(x)$ is a central polynomial (traces disappear). So taking two such elements of the same degree one   has 
$$F(x):=\sum_i u_1(x)   A_i\otimes u_2(x) B_i=u_1(x)u_2(x)  f(x)(1,2)$$ and $F(x)$ is a balanced swap polynomial. Moreover $u_1(x)u_2(x)  f(x)$  is a central polynomial.\smallskip

\subsection{Dual bases and  Capelli polynomials\label{Cap}}

The previous general procedure  is more effective using the approach to central polynomials of Razmyslov, \cite{razmyslov2} i.e. {\em antisymmetry} as follows.  

Given a noncommutative polynomial $f$ in several variables which is linear in a given variable $x_i$, write it in the form
$f=\sum_ka_kx_ib_k$. Consider it as a function on matrices,
set $  f_i:=\sum_kb_ka_k$ we then have:
$$ tr(f)=tr(x_if_i).$$

If $f=\sum_ka_kx_ib_k$   depends linearly also upon another variable $x_j$ and   it  changes sign by exchange of $x_i,x_j$, then when we substitute in $f$,
$x_i$ with $x_j$ we get $\sum_ka_kx_jb_k=0$ and  we also deduce:
$$ tr(x_jf_i)=tr(\sum_ka_kx_jb_k)=0.$$

Therefore if $f(x_1,\ldots,x_{n^2},y)$ is a noncommutative polynomial   which is linear in each variable $x_i$ and alternating in 
$x_1,\ldots,x_{n^2}$ and $tr(f(x_1,\ldots,x_{n^2},y)) $ is different from 0 we have $tr(x_jf_i)=\delta^i_jtr(f(x_1,\ldots,x_{n^2},y)) $. Thus we have, up to the scalar function $tr(f(x_1,\ldots,x_{n^2},y)) $ the polynomials $f_i$ form a dual basis of the $n^2$ variables $x_j$. From Formula   \eqref{ledua} we  also have that
$$\mathtt t=tr(f(x_1,\ldots,x_{n^2},y))^{-1}\sum_{i=1}^{n^2}x_i\otimes f_i,\quad\text{the Goldman element} $$        
\begin{equation}\label{adua0}
\implies h(x,y):=\sum_{i=1}^{n^2} x_i y_0  f_i(x_1,\ldots,\check x_i,\ldots,x_{n^2},y)=tr(y_0) tr(f(x ,y))  .
\end{equation} Thus we have a central polynomial $h(x,y)$ of degree $\deg(f)+1 $  and, by Formula  \eqref{ledua}, the swap polynomial:
\begin{equation}\label{adua}
H:=\sum_{i=1}^{n^2} x_i\otimes  f_i(x_1,\ldots,\check x_i,\ldots,x_{n^2},y)=tr(f(x_1,\ldots,x_{n^2},y)) \mathtt t.
\end{equation}
This $H$ is of course unbalanced, but if $m$ is the degree of $f$ and we choose any central polynomial $c$ of degree $m-2$  one has that 
$$c\cdot H= \sum_{i=1}^{n^2} c\cdot x_i\otimes  f_i(x_1,\ldots,\check x_i,\ldots,x_{n^2},y)=c\cdot tr(f(x_1,\ldots,x_{n^2},y)) \mathtt t$$ is balanced.   At worst, if we cannot find such a central polynomial,  replacing $f$  with $\bar f:=f uzw$  of degree $m+3$ we may take as $c=h(x,y)$ of degree $m+1=(m+3)-2$. \smallskip

The simplest $f$  satisfying the previous conditions is 
 the {\em Capelli polynomial} $C_{n^2}$, of degree $2n^2-1$. Following Razmyslov\index{Razmyslov} for each $m$ one  sets
 $$C_m(x_1,x_2,\dots,x_{m};y_1,y_2,\dots,y_{m-1} ) $$
\begin{equation}\label{Capelli}:=\sum_{\sigma\in
S_{m}}\epsilon_\sigma  x_{\sigma(1)}y_1x_{\sigma(2)}y_2\dots
x_{\sigma(m-1)}y_{m-1}x_{\sigma(m)}  .\end{equation}
In fact this is only an analogy of the classical Capelli identity, which is instead an identity of differential operators (see \cite{P7}).

Thus we have an explicit central polynomial of degree $2n^2$  and an explicit balanced swap polynomial of degree  $4n^2+2$.

%

 We will see later, in \S \ref{swap}, how to lower this degree and at the same time constructing more canonical swap polynomials.

%
%

  \subsection{Antisymmetry}Let us recall some basic facts. Denote for simplicity by $G=GL(n,F)$ the group of invertible matrices which acts on $n\times n$ matrices by conjugation.\begin{theorem}
The invariants of $n \times n$ matrices are generated by elements $\tr(M)$ where $M$ are monomials 
(of degree~$\leq n^2$ by Razmyslov). 
\end{theorem}
Among these invariants the ones that are multilinear and alternating have a very special structure.\smallskip

In fact these invariants have an {\em exterior} multiplication. The algebra of these invariants, 
under exterior  multiplication,   is the algebra of invariant multilinear alternating 
functions $(\bigwedge M_d(F)^*)^G$. 

In turn this algebra can be identified to the cohomology of the unitary group. 
As all such cohomology algebras it is a Hopf algebra and  by Hopf's Theorem  it  is the exterior 
algebra generated by the  primitive elements. 
 
The  primitive elements  of   $(\bigwedge M_d(F)^* )^G$ are~\cite{kostant}:

\begin{equation}\label{prime}
T_{2i-1}=T_{2i-1}(x_1,\ldots,x_{2i-1}):\stackrel{\eqref{st}}=\tr (St_{2i-1}(x_1,\ldots,x_{2i-1}))\,.
\end{equation} 
Recall the {\em standard polynomial} in $k$ variables is defined as
\begin{equation}\label{st}
 St_k(X) := \sum_{\sigma \in S_k} \epsilon_\sigma x_{\sigma(1)} \cdots x_{\sigma(k)}=Alt_Xx_1x_2\ldots x_k.
\end{equation}
$S_k$ denotes the symmetric group and $ \epsilon_\sigma $ the sign of the permutation $\sigma$.

In particular, since these elements generate an exterior algebra, 
a product of elements $T_i$ is non zero if and only if the $T_i$ involved are all distinct. Given $k$ distinct $T_i$ their product depends on the order only 
up to  sign.

The $2^n$ different products form a basis of  $(\bigwedge M_d(F)^* )^G$.   
In dimension $d^2$ the only non-zero product of these elements containing $d^2$ variables is
 \begin{equation}\label{ILTD}
\mathcal T_d(x_1,x_2,\ldots,x_{d^2})=T_1\wedge T_3\wedge T_5\wedge \cdots \wedge T_{2d-1}.
\end{equation} Notice  that $\tr(St_{2i }(x_1,\ldots,x_{2i}))=0,\ \forall i$.\smallskip

As a consequence, we have:
\begin{proposition}\label{mmu}
Any multilinear  antisymmetric function of  $x_1,\ldots,x_{d^2 }$   is a multiple of \
$T_1\wedge T_3\wedge T_5\wedge \cdots \wedge T_{2d-1}$.
\end{proposition} 

\begin{remark}\label{Iin} 
The function $\det( x_1,\ldots, x_{d^2})$  is an alternating invariant of matrices,
so it must have an expression as in Formula~\eqref{ILTD}.  
In fact the computable integer constant is known up to a sign~\cite{formanek2}:
\begin{equation}\label{costdis}\mathcal T_d
(x_1, \dots, x_{d^2})=\mathcal{C}_d\det( x_1,\ldots, x_{d^2}),\quad 
\mathcal C_d:=\pm \frac{1!3!5!\cdots (2d-1)!}{1!2!\cdots (d-1)!}\in \mathbb Z.
\end{equation}
 \begin{equation}\label{ilfa}
\pm\{1, 6, 360, 302400, 4572288000, 1520925880320000,\ldots \}
\end{equation}
\end{remark}
\begin{proposition}\label{crit}
\begin{enumerate}\item There is an element $A_G\in M_d(F)^{\otimes n}$ invariant under the diagonal action of the linear group $G=GL(d)$ such that 
\begin{equation}\label{JG}
G(X_1,\ldots,X_{k})=\prod_{i=1}^k \mathcal T_d
(X_i )A_G,\quad A_G\in \Sigma_n(F^d)\subset M_d(F)^{\otimes n} . 
\end{equation}\item    We have $$tr(\sigma\circ G(X_1,\ldots,X_{k})=\prod_{i=1}^k \mathcal T_d
(X_i )tr(\sigma\circ A_G),\ \forall \sigma\in S_n.$$ 
\item If $n=2$ then $G(X_1,\ldots,X_{k})\neq 0$ is a  swap polynomial if and only if:  \end{enumerate}
\begin{equation}\label{cott}
d\cdot tr(G(X_1,\ldots,X_{k}))=  tr((1,2)G(X_1,\ldots,X_{k}))\iff d\cdot tr(A_G )=  tr((1,2)A_G).
\end{equation}
\end{proposition}
\begin{proof} 
The first two parts are clear. As for   3. we have $A_G=a\cdot Id+b(1,2)$ and  $G(X_1,\ldots,X_{k})$ is a  swap polynomial if and only if $a=0$  (and $b\neq 0$):
$$  tr(A_G )= tr(a\cdot Id+b(1,2))=  a\cdot  d^2+b\cdot d  $$$$ tr((1,2)A_G)= tr(a(1,2) +b\cdot Id)=  a  \cdot  d+b\cdot d^2 $$ so  $$ a  \cdot  d+b\cdot d^2 = d(a\cdot  d^2+b\cdot d)\iff a=0.$$
 \end{proof}

  Let us also remark:
  \begin{remark}\label{qzer}
$$tr(a\cdot Id+b(1,2))=0 \iff a=-d\cdot b,\quad tr((1,2)(a\cdot Id+b(1,2))=0 \iff b=-d\cdot a.$$
\end{remark} Of course for a 2-tensor valued such polynomial $G(X,Y)$ we have that $G(X,Y)=0$ if and only if $ tr(G(X,Y))=tr((1,2)G(X,Y))=0$.

A method to compute $A_G $ is by using the transform  $\Phi_d$  introduced by Collins \cite{Coll} and the so called Weingarten function, which with a different terminology was already studied by  Formanek \cite{formanek2}.
 \begin{equation}\label{CW}
\Phi_d(A_G):=\sum_{\sigma\in S_n}\tr(\sigma\circ A_G)\sigma^{-1}, \ W\!g(n,d):=\Phi_d(\one)^{-1} \implies A_G=\Phi_d(A_G)W\!g(n,d).
\end{equation}
\begin{remark}\label{Univ}
One has that the element $\Phi_d(\one)$ is invertible in  the center of  $ \Sigma_n(F^d)$. Hence  $W\!g(n,d)= \sum_{\mu\vdash d} a_\mu c_\mu\in \Sigma_n(F^d)$    is  the image of a class function ($c_\mu$ denotes the sum over the conjugacy class relative to the partition $\mu$).  Of course the expression is unique only if  $d\geq n$, in this case one may interpret the Weingarten function as a function of $\sigma\in S_n$ depending on $d$. 
\end{remark}

An explicit formula through characters is, see \cite{Coll}, or \cite{P8}
 \begin{equation}\label{wein}
a_\mu=\sum_{\lambda\vdash n,\ ht(\lambda)\leq d} \frac{ \chi^\lambda(1)^2\chi_\lambda(\mu)}{s_{\lambda,d}(1)} .
\end{equation}Where  $\mu$ is the cycle partition of $\sigma$,  $\chi_\lambda(\mu)$ the character of $\sigma$ in the irreducible representation of $S_n$  corresponding  to $\lambda$ and  $s_{\lambda,d}(1)$ is the dimension of the corresponding irreducible representation of $GL(d,F)$.

\begin{remark}\label{pone}
It is known \cite{Novak}   (see also \cite{P8}  Theorem  1. 29)   that the function $a_\mu$     is always nonzero and positive if $\mu$ is the cycle partition of an even permutation and negative for odd permutations.
\end{remark}

For a    computation using Mathematica of  the   list   $ d!^2\sum_{\mu\vdash d} a_\mu c_\mu$ and $d\leq 8$, for $n=d$  see \cite{P8}.

\begin{remark}\label{prs}
In Proposition 16 of \cite{hp}  we have shown  that, if we distribute the $d^2$  variables $Y$ in $k$ monomials $n_i(Y)$ each of degree $h_i$  (with $\sum_ih_{i=1}^k=d^2$) then $Alt_Y  n_1(Y)\otimes \dots\otimes  n_k(Y))$ is 0 unless  the $h_i$ are a permutation of a refinement of the sequence $\delta_d:=1,3,\ldots, 2d-1$. In this case we have $tr(\sigma^{-1}G_d(Y_1,\ldots,Y_{d^2})=0$  unless $\sigma $  {\em glues together} the monomials so to recover the partition $\delta_d$. In this  case  $tr( \sigma^{-1}G_d(Y_1,\ldots,Y_{d^2}))=\pm \mathcal T_d(Y) .$ The sign is that of the permutation that $\sigma$ induces on the subset of the indices $i$ for  which $h_i$ is odd.

In particular if $k=d$ this means that   the $h_i$ are a permutation of  the sequence $\delta_d$.  For the sequence $\delta_d$    it follows $$n_i(Y)=y_{(i-1)^2+1}\ldots y_{i^2} ,\quad G_d(Y_1,\ldots,Y_{d^2}) :=  Alt_Y(n_1(Y)\otimes \dots\otimes  n_d(Y))$$  \begin{equation}\label{wstpw}
tr(\sigma\circ G_d(Y_1,\ldots,Y_{d^2}))=\begin{cases}
 \mathcal T_d(Y)\quad\text{if}\quad \sigma=\one\\0\quad\text{otherwise}.
\end{cases}
\end{equation} That is  $\Phi_d(G_d(Y_1,\ldots,Y_{d^2}))=  \mathcal T_d(Y)\one  $ and, Proposition 26 of  \cite{hp} gives:  \begin{equation}\label{forgz1}
G_d(Y_1,\ldots,Y_{d^2}):=  Alt_Y(n_1(Y)\otimes \dots\otimes  n_d(Y))=  \mathcal T_d(Y) W\!g(d,d).
\end{equation}

\end{remark} 

 \subsection{A construction of swap polynomials }

   Our final construction of swap polynomials is based on Proposition \ref{crit}. Suppose  we have two  tensor polynomials $G_i( X_1,\ldots,X_{k})\in F\langle X\rangle^{\otimes 2},\ i=1,2$ as in Formula \eqref{JG}
   \begin{equation}\label{JG1}
G_i(X_1,\ldots,X_{k})=\prod_{j=1}^k \mathcal T_d
(X_j )A_i,\quad A_i  \in \Sigma_2(F^d)\subset M_d(F)^{\otimes 2},\ i=1,2 . 
\end{equation}
We then have $\ A_i =a_iId+b_i(1,2),\ i=1,2. $ 
 
 Then   we may assume $a_i\neq 0,\ i=1,2$ otherwise one is already a swap polynomial.  We clearly have:\begin{theorem}\label{sswa}
 $$-a_2G_1+a_1G_2=\mathcal T_d(X)\mathcal T_d(Y)(-a_2b_1+a_1b_2)(1,2)$$  
 is  a swap polynomial (provided that  $(-a_2b_1+a_1b_2)\neq 0$) that is the two polynomials are not proportional.
\end{theorem}
%
%
 
 So the issue is to find a pair of such polynomials. First it is easy to see, cf.  \cite{hp}, that if $k=1$  a multilinear  balanced antisymmetric 
 tensor polynomial   $G(x_1,\ldots,x_{d^2})\in  M_d(F)^{\otimes 2},\ d\geq 2$ vanishes when computed in $d\times d$ matrices, so the minimum number of sets $X_i$ to consider is 2. The approach to find two such polynomials (balanced) for $k=2$, rests on the Theory of Formanek developed to prove Theorem \ref{teF}. 

In principle a simple way of finding  a  polynomial   multilinear and alternating in each of   two sets of $d^2$  variables $X:=(x_1,\ldots,x_{d^2});  \ Y=(y_1,\ldots,y_{d^2})$ is the following.
Take any monomial $M(X,Y)$ product in some order of the $2d^2$  variables $X,Y$  and alternate  $F_M(X,Y):=Alt_XAlt_YM(X,Y)$. 

Also split $M(X,Y):=AB$ with $A,B$ each of length $d^2$, and alternate  $G_M(X,Y):=Alt_XAlt_YA\otimes B$. The real issue is to choose   $M(X,Y)$ so that $G_M(X,Y)$ is not a tensor polynomial identity. Of course if $F_M(X,Y)$ is not a   polynomial identity  then $G_M(X,Y)$ is not a tensor polynomial identity \cite{hp}.  \smallskip
   
A theorem of Formanek relative to a conjecture of Regev, see  \cite{formanek2}, states that a certain explicit polynomial $F(X,Y)$ in  $d^2$ variables $X=\{x_1,\ldots,x_{d^2}\}$ and    $d^2$ variables $Y=\{y_1,\ldots,y_{d^2}\}$ is non zero. A general discussion can be found in theorems of Giambruno Valenti  \cite{giambruno.valenti}.

From the Regev polynomial  we shall construct two tensor polynomials with the required properties.
   
   The definition of $F(X,Y)$ is this. Decompose $d^2=1+3+5+\ldots +(2d-1)$ and accordingly decompose the $d^2$ variables $X $ and  the $d^2$ variables $Y$ in the list  and construct the monomials  $m_i(X), i=1,\ldots,d $    and similarly  $m_i(Y) $ as
   $$m_i(X)=x_{(i-1)^2+1}\ldots x_{i^2},\quad m_i(Y)=y_{(i-1)^2+1}\ldots y_{i^2}.$$ \begin{equation}\label{imono}
m_1(X)= x_1, m_2(X)= x_2x_3x_4 ,  m_3(X)= x_5x_6x_7x_8x_9 ,\ldots .
\end{equation}We finally define\begin{equation}\label{RFa}
F(X,Y):= Alt_XAlt_Y(m_1(X)m_1(Y)m_2(X)m_2(Y)\ldots m_d(X)m_d(Y)),
\end{equation}  where $Alt_X$ (resp. $Alt_Y$) is the operator of alternation in the variables $X$ (resp. $Y$). 
\begin{theorem}\label{teF}[see \cite{formanek2} or \cite{P8}]
\begin{equation}\label{RFaFF}
F(X,Y) = (-1)^{d-1} \frac  {1}{(d!)^2 (2d-1)}\mathcal T_d(X)\mathcal T_d(Y) Id_d
\end{equation} 
$$\stackrel{\eqref{costdis}}= (-1)^{d-1} \frac  {C_d^2}{(d!)^2 (2d-1)}\Delta(X)\Delta(Y) Id_d;\quad \Delta(X)=\det(x_1,\ldots,x_{d^2}).$$
\end{theorem}
In fact  this follows from the special value for $\sigma=(1,2,\ldots,d)$,   the full cycle,  of the Weingarten function:
$$a_\sigma= (-1)^{d-1} \frac  {d}{(d!)^2 (2d-1)} $$
Thus $F(X,Y) $ is a central polynomial. In fact it has also the property of being in the {\em conductor } of the ring of polynomials in generic matrices inside the trace ring. In other words  by multiplying  $F(X,Y) $ by any invariant we still can write this as a non commutative polynomial. This follows  by polarizing in $z$  the identity, cf    Propositions  10.2.10 and 10.215 or Theorem 10.4.8 of \cite{agpr} 
$$\det(z)^dF(X,Y) =F(zX,Y) =F(X,zY) .$$
\smallskip

In order to use this  for tensor polynomials 
we start from a general fact. 

Let us consider two decompositions  of $d^2$ as a sum of $c\leq d$ positive integers:
 \begin{equation}\label{hek}
\underline h:=(h_1,\ldots,h_c );\quad \underline k:=(k_1,\ldots,k_c )\mid \sum_ih_i=\sum_ik_i=d^2.
\end{equation} 
Decompose accordingly the list $X$ (resp $Y$) in $c\leq d$ lists with the $i^{th}$  list  formed by the $h_i$  (resp $k_i$)   variables successive to the ones of the previous lists.  Define $m_i(X)$  the ordered product of the variables in the $i^{th}$  list relative to $X$ similarly $n_j(Y)$ for $Y$. We have $m_i(X)$ has degree $h_i$ and   $n_j(Y)$ has degree $k_j.$
Define next $N_i=N_i(X,Y):=m_i(X)n_i(Y),\ i=1,\ldots,c$.
$$ N_1\otimes N_2\otimes \ldots \otimes N_{d}=m_1(X)  \otimes m_2(X) \otimes \ldots \otimes m_d(X) \cdot n_1(Y)  \otimes n_2(Y) \otimes \ldots \otimes n_d(Y) .$$
From Remark \ref{prs} it follows that
 the element    $$Alt_XAlt_Y( N_1\otimes N_2\otimes \ldots \otimes N_{d})$$  is 0 unless both  $\underline h$ and $  \underline k$ are permutations $\eta,\zeta$ of the sequence $1,3,5,\ldots,2d-1$.  
 
 In this last case  always by the same remark since $$Alt_Y(n_1(Y)  \otimes n_2(Y) \otimes \ldots \otimes n_d(Y)) =\pm \mathcal T_d(Y)W\!g(d,d) =\pm \mathcal T_d(Y)\sum_{\tau\in S_d}  a_\tau\tau $$ we have by Formula  \eqref{wstpw}, for $\sigma\in S_d$:
\begin{equation}\label{getttr0}
\pm  Alt_XAlt_Y  tr(\sigma^{-1}\circ N_1\otimes N_2\otimes \ldots \otimes N_{d}) 
\end{equation}  \begin{equation}\label{getttr}
=\mathcal T_d(Y)\sum_\tau  a_\tau tr(\sigma^{-1}\circ Alt_Xm_1(X)  \otimes m_2(X) \otimes \ldots \otimes m_d(X) \circ \tau)
\end{equation}$$\stackrel{  \eqref{wstpw}}=\pm \mathcal T_d(Y)\mathcal T_d(X)a_\sigma.$$
 The sign is  $\epsilon_\eta\epsilon_\zeta$
if  $\underline h$ and $  \underline k$ are permutations $\eta,\zeta$ of  $1,3,5,\ldots,2d-1$.   

\medskip

%
%

Take a monomial $M(X,Y)$ product in some order of the $2d^2$  variables $X,Y$.  
If we split $M(X,Y)=AB$ as a product of two factors each of length $d^2$  we may construct  the 2--tensor valued polynomial
 \begin{equation}\label{eee}
F_M^{\otimes 2} (X,Y):=Alt_XAlt_YA\otimes B
\end{equation}

If $F_M(X,Y)$ is not a polynomial identity we can take $G_1(X,Y)=F_M^{\otimes 2} (X,Y)$  from the previous Formula.

Instead, if $F_M(X,Y)$  is   a polynomial identity for $d\times d$ matrices, then $tr((1,2)F_M^{\otimes 2} (X,Y))=tr(  F_M  (X,Y))=0$ so 
$F_M^{\otimes 2} (X,Y) $ can be taken as  $G_2(X,Y)$ provided $tr( F_M^{\otimes 2} (X,Y))=Alt_XAlt_Ytr(A)tr(B)\neq 0.$

For the first we can take the polynomial of Formula \eqref{RFaFF}, we need to find the second.

\subsubsection{A construction  for $d=2h$   even\label{swap}  }
Let us use Formula \ref{getttr}  to construct two polynomials $G_1,G_2$ satisfying  the hypotheses of Theorem \ref{sswa} when $d=2h$ is even.

Consider the three monomials products of $d$ factors  in the variables $X,Y$
$$A= m_1(X)m_2(Y)m_3(X)\ldots m_{ d-1 }(X) m_{ d }(Y), $$
$$B= m_{ d }(X)m_{ d-1 }(Y)\ldots m_3(Y) m_2(X)m_1(Y) .$$
$$C=m_1(Y) m_{ d }(X)m_{ d-1 }(Y)\ldots m_3(Y) m_2(X) .$$ All of degree $d^2$ and $A, B $ or $A,C$ involving disjoint sets of variables.
Set  \begin{equation}\label{ilsb}
G_1(X,Y)=Alt_XAlt_YA\otimes B,\quad G_2(X,Y)=Alt_XAlt_YA\otimes C.
\end{equation} Both are multilinear balanced tensor polynomials in the variables $X,Y$.  

We have  $tr(B)=tr(C)$ and $AC=A'm_{d }(Y)m_{ 1}(Y)C'$ so:
$$ Alt_XAlt_Y A  B =F(X,Y)\neq 0,   Alt_Ym_{d }(Y)m_{ 1}(Y)=St_{2d}(Y)=0\implies Alt_XAlt_Y A  C = 0.$$
Apply Formula \ref{getttr}  to the $N_i$  decomposing $AB$ and $AC$;
  \begin{equation}\label{ilcon0}
tr(  G_2(X,Y))=tr(G_1(X,Y))=Alt_XAlt_Ytr(A)tr( B)\stackrel{\eqref{getttr}}= 
\mathcal T_d(X)\mathcal T_d(Y) a_{h,h}   .\end{equation}
\begin{equation}\label{ilcon}
tr((1,2)\circ G_1(X,Y))=Alt_XAlt_Ytr(A  B)\stackrel{\eqref{getttr}}=\mathcal T_d(X)\mathcal T_d(Y) a_{d} .
\end{equation}
$$tr((1,2)\circ G_2(X,Y))=0. $$
 
$$\text{Set}\ A_i:=A_{G_i}=a_iId+b_i(1,2),\ i=1,2. $$
$$  tr(A_i )= a_i\cdot  d^2+b_i\cdot d,\quad  tr((1,2)A_i = a_i  \cdot  d+b_i\cdot d^2 .$$ 
From Formulas  \eqref{ilcon0}  and \eqref{ilcon} 
$$a_2d+b_2d^2=0,\    a_1d+b_1d^2= a_{d} ,\quad a_2\cdot  d^2+b_2\cdot d=  a_1\cdot  d^2+b_1\cdot d= a_{h,h}. $$
Therefore solving the system of linear equations we have:
$$a_1=\frac{a_d-d\cdot a_{h,h}}{  d(1-d^2)},\quad b_1=-\frac{ d\cdot a_d-a_{h,h}}{  d(1-d^2)},\quad  b_2= \frac{  a_{h,h}}{  d(1-d^2)} ,\quad   a_2=- \frac{   a_{h,h}}{   (1-d^2)}. $$
We want to use Formula \eqref{cott} and hence
\begin{theorem}\label{swaa}
$A_G=-a_2A_1+a_1A_2$  is a balanced swap polynomial         with value $  \mathcal T_d(X)\mathcal T_d(Y) ( - a_2b_1+a_1b_2)(1,2)$.\begin{equation}\label{finn}
d\cdot  a_{h,h}G_1(X,Y)+(a_d-d\cdot a_{h,h})G_2(X,Y)=\frac{  a_{h,h} (d\cdot a_d-a_{h,h})}{(1-d)(1-d^2)d!^2   } \mathcal T_d(X)\mathcal T_d(Y) (1,2).
\end{equation}

\end{theorem}   For $d=2,4,6$ we have, after multiplying by $d!^2$ that $$a_2=-\frac 23,\ a_4=-\frac 47,\ a_6=-\frac 6{11},\quad a_{1,1}=\frac{4}{3 },\quad a_{2,2}=\frac{22}{35},\quad a_{3,3}=\frac{300}{539}.$$
 $$  d=2,\  \frac 83G_1- \frac {10}3G_2=-\frac8{27}\mathcal T_2(X)\mathcal T_2(Y) (1,2) \stackrel{\eqref{ilfa}}= -\frac{32}3 D(X)D(Y) (1,2)$$$$ 4G_1-5G_2=  -16 D(X)D(Y) (1,2);$$
$$  d=4,\ 22G_1-27G_2= -\frac{561}7D(X)D(Y) (1,2); $$$$  d=6,\ 75 G_1 - \frac{3893}{ 44} G_2 = -\frac {310564800}{6523}D(X)D(Y) (1,2).$$ 
\begin{remark}\label{manc}
 By Remark \ref{pone}  have that $a_{h,h}>0$  while $a_d<0$, so the coefficient $\frac{  a_{h,h} (d\cdot a_d-a_{h,h})}{(1-d)(1-d^2)d!^2   }$ is $<0$ and the tensor polynomial is thus nonzero.
\end{remark}
                                 
\subsubsection{The case $d$ odd} The previous construction does not apply to the case $d$  odd.  In this case
$$A= m_1(X)m_2(Y)m_3(X)\ldots m_{  d-1 }(Y) m_{ d }(X), $$
$$B= m_{  d }(Y)m_{ d-1 }(X)\ldots m_3(Y) m_2(X)m_1(Y) .$$
 $$tr(Alt_X A)= tr(Alt_X m_1(X)m_ { d}(X)m_2(Y)m_3(X)\ldots m_{ d-1 }(Y)=0 $$ since $Alt_X m_1(X)m_ { d}(X)=St_{2d}(X)=0$.

We construct $G_1(X,Y)=Alt_XAlt_YA\otimes B$ and have $tr(G_1(X,Y))=0,$ $   tr((1,2)\circ G_1(X,Y))= \mathcal T_d(X)\mathcal T_d(Y) a_{d} .$ 
$$G_1(X,Y)=\mathcal T_d(X)\mathcal T_d(Y)(a+(1,2)b,\   ad^2+bd=      
0,\   ad+bd^2=  a_d   $$
$$\implies b=-\frac{a_d}{ (1-d^2)}, \  a=\frac{a_d}{d(1-d^2)}. $$ We need another tensor polynomial $G_2(X,Y) $  with  $tr(G_2(X,Y))\neq 0$. In fact  we will give one with $tr((1,2)G_2(X,Y))= 0.$

Let us do it for $d=3$  and the general case is similar. We start from the monomials of Formula \eqref{imono}
$$m_1(X)= x_1, m_2(X)= x_2x_3x_4 ,  m_3(X)= x_5x_6x_7x_8x_9 . $$
We split the monomials $m_2$  and consider

$$A=x_1y_1x_3x_4m_3(Y),\quad B=y_2x_2y_3y_4m_3(X)$$
$$G(X,Y):=  Alt_XAlt_YA\otimes B,\ tr((1,2)  G(X,Y))=  tr(  Alt_XAlt_YA  B )    =0     $$ since in $AB$ appears the factor  $m_3(Y)y_2$  of degree 6 whose alternation is 0.
We need to show that  $ tr(  G(X,Y))= Alt_XAlt_Y tr(  A )tr( B )  $ is different from 0.
 Next  $tr(B)=tr(m_3(X)y_2x_2y_3y_4)$  and consider 
\begin{equation}\label{ilmon}
 M_1:= x_1 \otimes x_3x_4 \otimes m_3(X)   \otimes x_2,\ M_2:=  y_1\otimes  m_3(Y) \otimes   y_2\otimes   y_3 y_4
\end{equation}
$$ P=M_1M_2,\,\  tr\left((1,2)(3,4)P\right)=tr(A)tr(B)$$\begin{equation}\label{ilmon1}\implies  tr(  G(X,Y))= Alt_XAlt_Y tr\left((1,2)(3,4)P\right).\end{equation}
 \begin{lemma}\label{sem}
$
Alt_Xtr(\sigma^{-1}  M_1) =Alt_Ytr(\sigma^{-1}  M_2)  =0$ except for the following cases:
\begin{align*}
 tr((1,2)   x_1 \otimes x_3x_4 \otimes m_3(X)   \otimes x_2=&tr(x_2) tr(x_1 x_3x_4)tr(m_3(X))\\ tr((2,4)   x_1 \otimes x_3x_4 \otimes m_3(X)   \otimes x_2=&tr(x_1) tr(x_2 x_3x_4)tr(m_3(X))\\
 tr((1,4) y_1\otimes  m_3(Y) \otimes   y_2\otimes   y_3 y_4
)=&tr(y_2) tr(y_1 y_3y_4)tr(m_3(Y))\\ tr((3,4)y_1\otimes  m_3(Y) \otimes   y_2\otimes   y_3 y_4
)=&tr(y_1) tr(y_2 y_3y_4)tr(m_3(Y))
\end{align*}
\end{lemma}
\begin{proof}
 The other $\sigma$ give a product of trace monomials of lengths different from the partition $1,3,5$ hence give 0.\end{proof}
We deduce for $Alt_Y M_2 $ Formula \eqref{almu0}
\begin{equation}\label{altz}
Alt_Ytr((3,4)M_2)=\mathcal T_3(Y),\quad Alt_Ytr((1,4)M_2)=-\mathcal T_3(Y);  
\end{equation}
$$\Phi_d(Alt_YM_2)  =\mathcal T_3(Y)[(3,4)-(1,4)]$$
\begin{equation}\label{almu0}
\implies Alt_Y M_2   =\mathcal T_3(Y)[(3,4)-(1,4)]W\!g(4,3).
\end{equation}

We deduce for $Alt_X M_1 $ Formula \eqref{almu}
\begin{equation}\label{altz1}
Alt_Xtr((2,4)M_1)=\mathcal T_3(X),\quad Alt_Xtr((1,2)M_2)=-\mathcal T_3(X);  
\end{equation}
$$\Phi_d(Alt_XM_1)  =\mathcal T_3(X)[(2,4)-(1,2)]$$
\begin{equation}\label{almu}
\implies Alt_X M_1   =\mathcal T_3(X)[(2,4)-(1,2)]W\!g(4,3).
\end{equation} 
$$Alt_YP= x_1 \otimes x_3x_4 \otimes m_3(X)   \otimes x_2 \cdot  \mathcal T_3(Y) [(3,4)-(1,4)]W\!g(4,3) $$
Thus  $ tr(  G(X,Y))= Alt_XAlt_Y tr(  A )tr( B )  $  is  given, \eqref{ilmon1},  by $\mathcal T_3(Y)$ times:
$$Alt_X  tr\left((1,2)(3,4)   x_1 \otimes x_3x_4 \otimes m_3(X)   \otimes x_2 \cdot   [(3,4)-(1,4)]W\!g(4,3) \right). $$
 \begin{equation}\label{nzc}
=\sum_{\sigma\in S_4} b_\sigma Alt_X  tr\left((1,2)(3,4)   x_1 \otimes x_3x_4 \otimes m_3(X)   \otimes x_2 \cdot   [(3,4)-(1,4)] \sigma \right). 
\end{equation}
By Lemma \ref{sem} we know that we have non zero contributions $\pm  b_\sigma \mathcal T_3(X)$, to Formula  \eqref{nzc}, from $W\!g(4,3) =\sum_\sigma b_\sigma \sigma$ only when  
$$-  b_\sigma\ \text{if}\ (3,4) \sigma (1,2)(3,4) =(1,2)   ,\ b_\sigma\ \text{if}\  (3,4) \sigma (1,2)(3,4) =  (2,4) $$$$ b_\sigma\ \text{if}\ ( 1,4) \sigma (1,2)(3,4) =(1,2), - b_\sigma\ \text{if} \ ( 1,4) \sigma (1,2)(3,4) =(2,4) .  $$
The corresponding 4 values of $\sigma$ are:
\begin{align*} 
 &\sigma  =(3,4) (1,2)(3,4)(1,2)= 1,  &\text{sign}\ -\\
 &\sigma  =(3,4) (2,4) (1,2)(3,4)= (1,3,2  ),  &\text{sign}\  + \\
  & \sigma  =(1,4)(1,2)(1,2)(3,4)=(1,4,3),   &\text{sign}\ -\\
 & \sigma  =(1,4)(2,4)(1,2)(3,4) =  (2, 4,3),  &\text{sign}\  +   \\
\end{align*}
But $b_{\sigma}$ is a class function so the contribution is $-b_{1^4}+b_{3,1}$ and
$$tr(  G(X,Y))=Alt_XAlt_Y tr(  A )tr( B )  =\mathcal T_3(X)\mathcal T_3(Y)   [-b_{1^4}+b_{3,1}]$$
$$ (4!) ^2  [b_{2,2}-b_{3,1}] =- \frac{61}{5}-\frac{3 }{ 5} =-\frac{64 }{ 5}.$$
Now the expression of  $W\!g(4,3) =\sum_\sigma b_\sigma \sigma$ is not unique (cf. Remark \ref{Univ}) since $ \sum_\sigma \epsilon_\sigma \sigma=0$ but the difference $b_\sigma -b_\tau$ for two permutations of the same parity or the sum $b_\sigma +b_\tau$ for two permutations of   opposite parity is well defined. 

Examples of $b_\lambda\cdot (d+1)!^2$  for $d=2,3,4,5 $:
$$ -  \frac  74 c_3+\frac 1 4c_ {2, 1 } +\frac{17}{4}c_{1,1,1};\qquad \frac{37}{15}c_4- \frac{3}{5} c_{3, 1} - \frac{11}{5} c_{2, 2}- \frac{7}{3 }c_{2, 1, 1}  +\frac{61}{5}c_{1,1,1,1}$$ 
 $$  - \frac{533}{168}c_5  +-\frac{143}{168 }c_{4,1}   +\frac{503}{168}c_{3,2 }  +\frac{61}{
 24 }c_{3,1,1 }   -\frac{53}{168}c_{2,2,1 }   -\frac{1417}{168}c_{2, 1, 1,1 }+\frac{5227}{168}   c_{1,1,1,1,1}$$

$$   \frac{1627}{420}c_{6}  -\frac{451}{420 }c_{5,1}  -\frac{389}{105 }c_{4,2}  -\frac{104}{35 }c_{4,1,1}  -\frac{1601}{
  420 }c_{3,3} +\frac{151}{210}c_{3,2,1} +\frac{991}{105}c_{3,1,1,1} $$$$ +\frac{701}{210}c_{2,2,  2}  +\frac{289}{70}c_{2,2,1 , 1}  - \frac{4649}{210 } c_{2,1,1,1,1} +\frac{5227}{168}c_{1,1,1,1,1,1}$$

The general case $d+1=2h$ reduces to this, see Formula \eqref{laff},  by considering the two monomials
$$C=m_4(Y)m_5(X) \ldots m_{ d-1 }(Y) m_{ d }(X);   D=m_4(X)m_5(Y) \ldots m_{ d-1 }(X) m_{ d }(Y) $$  substituting 
$$Alt_YAlt_X AC\otimes BD,\ Alt_YAlt_X AC  BD=0$$
$$tr(AC )( BD)=tr(\tau_h^{-1} Q),\ \tau_h= (1,2,\ldots,h)(h+1,\ldots, 2h),\ Q=Q_1\cdot Q_2 $$
$$Q_1 = x_1 \otimes x_3x_4 \otimes m_3(X)   \otimes x_2\otimes m_4(X)\otimes m_5(X)\otimes  \ldots\otimes  m_{ d }(X)  $$$$Q_2 =  y_1\otimes  m_3(Y) \otimes   y_4 y_3 \otimes  y_2\otimes   m_4(Y)\otimes m_5(Y)\otimes  \ldots\otimes  m_{ d }(Y) $$ continuing in the same way we need  the non zero contributions $\pm  b_\sigma $ from $W\!g(d+1,d) =\sum_\sigma b_\sigma \sigma$ only when  
$$-  b_\sigma\ \text{if}\ (3,4) \sigma \tau_h^{-1} =(1,2)   ,\ b_\sigma\ \text{if}\  (3,4) \sigma \tau_h^{-1} =  (2,4) $$$$ b_\sigma\ \text{if}\ ( 1,4) \sigma \tau_h^{-1} =(1,2), - b_\sigma\ \text{if} \ ( 1,4) \sigma \tau_h^{-1} =(2,4) .  $$
The corresponding 4 values of $\sigma$ are:
$$\sigma  = (1,2)(3, 4)\tau_h =        ( 2,4,\ldots,h )( h+1,\ldots,2h  )\qquad\qquad\qquad\quad$$
$$\sigma  =(3,4) (2,4) \tau_h =(2,3,4) \tau_h =  ( 1,3,2,4,\ldots,h     )( h+1,\ldots,2h  )$$
$$  \sigma  =(1,4)(1,2)\tau_h =(1,2,4) \tau_h= (   2,3  ,1,4,\ldots,h  ) ( h+1,\ldots,2h  )$$
$$ \sigma  =(1,4)(2,4)\tau_h = (1,4,2 )\tau_h = ( 2,3)(  4,\ldots,h  )  ( h+1,\ldots,2h  ).  $$
 \begin{equation}\label{laff}
\boxed{tr(  G(X,Y))= \mathcal T_d(X)\mathcal T_d(Y)[-b_{ 1^2,h-2,h }+2b_{h,h}-b_{1,2,h-3,h }]}.
\end{equation}$$d=5,\ h=3,(5!)^2[ -b_{1^3,3}+2b_{3,3}-b_{1,2,3}]=-\frac{991}{105 }-2  \frac{1601}{
  420 } -\frac{151}{210 }=-\frac{1867}{105}   .$$
  It remains to prove that for all $h$ we have $-b_{ 1^2,h-2,h }+2b_{h,h}-b_{1,2,h-3,h }\neq 0.$ Unfortunately we cannot use directly Novak's argument, so we leave this open.
  
  In fact we  have seen   that the element $\Phi_d(\one)$ is invertible in  the center of  $ \Sigma_n(F^d)$ (cf. Remark \ref{Univ}). Novak's result follows from   writing   the inverse as a Neumann series depending on the parameter $d$ with terms in  $\mathbb C[S_n]$. Then the fact that  $ \Sigma_n(F^d)$ equals the group algebra $\mathbb C[S_n]$  only for $d\geq n$  reflects into the fact that this series converges only for $d\geq n$.  Maybe a closer look at Novak's proof and the spectrum of Jucy--Murphy elements can solve this problem (see \cite{P8}).
\bibliographystyle{amsalpha}

\end{document}